\documentclass[12pt,a4paper,leqno]{amsart} 
\usepackage[english]{babel}
\usepackage[T1]{fontenc}
\usepackage[utf8]{inputenc}
\usepackage{mathrsfs}
\usepackage{amsmath,amssymb,amsfonts,amsthm,amscd}
\usepackage{latexsym}
\usepackage{hyperref}
\usepackage{breakurl}
\usepackage{epigraph}
\usepackage{nicefrac}
\usepackage{enumitem}
\usepackage{graphicx}
\usepackage[lmargin=31mm,rmargin=31mm]{geometry}
\usepackage{tikz-cd}
\usepackage[babel]{csquotes}
\usepackage{bm}
\DeclareMathOperator{\Map}{Map}
\newcommand{\bfMap}{\bm{\mathrm{Map}}}
\DeclareMathOperator{\Ho}{Ho}
\DeclareMathOperator{\Hom}{Hom}
\DeclareMathOperator{\Aut}{Aut}

\DeclareMathOperator{\RHom}{\mathbb{R}\underline{Hom}}
\DeclareMathOperator{\RPerf}{\mathbb{R}Perf}
\DeclareMathOperator{\RSpec}{\mathbb{R}Spec}
\DeclareMathOperator{\hocolim}{hocolim}
\DeclareMathOperator{\holim}{holim}
\DeclareMathOperator{\colim}{colim}
\DeclareMathOperator{\ext}{ext}
\DeclareMathOperator{\Ext}{Ext}
\DeclareMathOperator{\supp}{supp}
\DeclareMathOperator{\ch}{ch}
\DeclareMathOperator{\td}{td}
\DeclareMathOperator{\Spec}{Spec}
\DeclareMathOperator{\Perf}{Perf}
\DeclareMathOperator{\Coh}{Coh}
\DeclareMathOperator{\QCoh}{QCoh}

\DeclareMathAlphabet\mathbfcal{OMS}{cmsy}{b}{n}  
\DeclareMathAlphabet{\mathpzc}{OT1}{pzc}{m}{it}  
\renewcommand{\lim}{\varprojlim}
\renewcommand{\H}{\mathrm{H}}
\newcommand{\calH}{\mathcal{H}}
\renewcommand{\P}{\mathbb{P}}
\newcommand{\id}{\mathrm{id}}

\newcommand{\Z}{\mathbb{Z}}
\newcommand{\Q}{\mathbb{Q}}
\newcommand{\R}{\mathbb{R}}
\newcommand{\C}{\mathbb{C}}
\newcommand{\G}{\mathbb{G}}
\renewcommand{\L}{\mathbb{L}}

\newcommand{\A}{\mathcal{A}}
\newcommand{\bfcalC}{\mathbfcal{C}}
\renewcommand{\O}{\mathcal{O}}
\newcommand{\calP}{\mathcal{P}}
\newcommand{\rmD}{\mathrm{D}}
\newcommand{\rmL}{\mathrm{L}}

\newcommand{\scrM}{\mathscr{M}}

\let\phi\varphi
\renewcommand{\epsilon}{\varepsilon}

\mathchardef\dsh="2D
\newcommand{\catname}[1]{{\normalfont\textbf{#1}}}

\newcommand{\sSet}{\catname{sSet}}
\newcommand{\CAlg}{\catname{CAlg}}
\newcommand{\sCAlg}{\catname{sCAlg}}
\newcommand{\Ch}[1]{\catname{Ch(}#1\catname{)}}
\newcommand{\dgCh}[1]{\catname{Ch}_{dg}\catname{(}#1\catname{)}}
\newcommand{\Mod}{\catname{-Mod}}
\newcommand{\dgCat}{\catname{dgCat}}
\newcommand{\dgVect}{\catname{dgVect}}
\newcommand{\cdgAlg}{\catname{cdgAlg}}
\newcommand{\dSt}{\catname{dSt}}
\newcommand{\Lotimes}{\otimes^\mathbb{L}}
\newcommand{\wh}{\widehat}
\newcommand{\whperf}[1]{{\wh{#1}}_{perf}}

\newcommand{\nfrac}{\nicefrac}
\newtheorem{prop}{Proposition}
\newtheorem{dfn}{Definition}
\newtheorem{thm}{Theorem}

\newtheorem{lemma}{Lemma}
\newtheorem*{thm*}{Theorem}

\theoremstyle{remark}
\newtheorem{rmk}{Remark}
\newtheorem*{ex}{Example}
\author{Martino Cantadore}
\title{Dg categories of cubic fourfolds}
\date{\today}

\begin{document}
\maketitle

\begin{abstract}
We prove a reconstruction theorem à la Calabrese-Groechenig for the moduli space parametrizing skyscraper sheaves on a smooth projective variety when these are considered as a system of points in the dg category of perfect complexes on the variety, as axiomatized by To\"en and Vaquié. This result is then used to show that, for a cubic fourfold $Y\subset \P^5_\C$, the Kuznetsov category $\A_Y$ is geometric (possibly twisted) if and only if a dg enhancement $T_Y$ of $\A_Y$ admits a system of points whose associated moduli space is a (possibly twisted) K3 surface.
\end{abstract}

\section*{Introduction}
In recent years some interest has grown around the possibility of studying rationality problems of varieties by means of derived category techniques. One of the most stunning approaches is the cautious conjecture made by Kuznetsov \cite{kuznetsov} stating that a smooth cubic hypersurface $Y\subset \P^5_\C$ is rational precisely if the subcategory $\mathcal{A}_Y$ appearing in the semi-orthogonal decomposition (see subsection \ref{cfdaK3} for the precise definition)
\[
\rmD^b(\Coh Y)=\langle \A_Y,\O_Y,\O_Y(1),\O_Y(2)\rangle \qquad\A_Y:=\langle \O_Y,\O_Y(1),\O_Y(2)\rangle^\bot
\]
is equivalent to the derived category of coherent sheaves over a K3 surface.

The category $\A_Y$ is an instance of what some people call \emph{noncommutative variety}, that is a $k$-linear triangulated category with a Serre functor and such that the $\Hom$ spaces between its objects are finite-dimensional. In this mindset we should do geometry by studying, instead of varieties themselves, triangulated categories which have the same properties of derived categories of coherent sheaves over varieties. In this case $\A_Y$ is an example of a K3 category because its Serre functor is the double shift like in the derived category of a K3 surface.

The rough idea behind this paper is that we would like to treat the triangulated category $\mathcal{A}_Y$ as a category of geometric origin: to do that we  first take a dg enhancement of it, that is we see it as a triangulated trace of a saturated dg category $T$, and we can then use the results of To\"en and Vaquié \cite{toenvaquie} to produce a derived moduli stack $\R\scrM_T$ which classifies perfect (i.e. compact) objects in $T$. The next step would be looking inside this moduli stack to find a suitable geometric space which would be our deputee K3 surface.

K3 surfaces have non-trivial Fourier-Mukai partners, so that there is no possibility to reconstruct such a K3 surface $S$ from the datum of its derived category alone, nevertheless the construction mentioned above gives us back the surface (and in fact any smooth projective variety) when we apply it in the case $\A_Y\simeq \rmD^b(S)$ because we are given a canonical morphism from $S$ to the moduli space of perfect complexes over $S$ as the unit of an adjunction (and in particular we use the t-structure on the derived category).
Our approach to the reconstruction follows the paper \cite{toenvaquiepoints} where the authors axiomatize what a point in a dg category is and construct a moduli stack $\scrM_\calP\subset \scrM_T$ of point-like objects in a given dg category for a given system of points $\calP$ (see section \ref{secpts} for a resumé of their work). The moduli stack of point-like objects results naturally to be a $\G_m$-gerbe over the coarse moduli space $M_\calP$.

In the same paper the authors claim that skyscrapers sheaves on a smooth projective variety constitute a system of points and that a smooth projective variety can be reconstructed starting from such data, in this work we provide a proof of this fact and apply this reconstruction procedure to get the following result: 
\begin{thm*}[see Theorem \ref{thm} and Proposition \ref{mtpt0}]
Let $Y\subset \P^5$ be a cubic fourfold, $\A_Y$ the admissible triangulated subcategory in the semi-orthogonal decomposition of $\rmD^b(Y)$ as above and $T_Y$ a dg enhancement of the latter. The following are equivalent:
\begin{enumerate}
\item There exists a twisted K3 surface $(S,\alpha)$ such that $\rmD^b(S,\alpha)\simeq \A_Y$;
\item There exists a system of points $\calP$ of dimension two in $T_Y$ strongly co-generating a t-structure on $T_Y$ and $M_\calP$ is a K3 surface.
\end{enumerate}
In particular the K3 is untwisted (i.e. $\alpha=1$) if and only if $\scrM_\calP$ is the trivial $\G_m$-gerbe on $M_\calP$ and in this case there exists an open immersion $S \hookrightarrow M_T^{pt,0}$.
\end{thm*}

In the statement above the space $M_T^{pt,0}$ is the coarse moduli space of the $\G_m$-gerbe $\scrM^{pt,0}_T$ (see subsection \ref{pt0} for the definition) and is an attempt to capture the information of point-like objects by considering the open substack of $\scrM_T$ made of objects in $T_Y$ which have the correct ext-groups and Chern character, but we were unable to prove a statement as strong as the one obtained for $\scrM_\calP$.
\\

This paper is structured as follows: after a section of preliminaries, in Section 2 we establish the reconstruction result for general smooth projective variety. In Section 3 we apply the result to the case of cubic fourfolds and K3 surfaces.

\section{Preliminaries}
We are going to use some technical tools from higher/dg category theory and derived geometry, as well as geometric facts about varieties, especially (smooth projective) K3 surfaces. Since it is more common to feel comfortable on either side but rarely on both, we do our best to summarize the necessary facts in this section.

\subsection{Cubic fourfolds, decompositions and associated K3's}\label{cfdaK3}

It has been remarked that cubic fourfolds, that is smooth hypersurfaces in $\P^5_\C$, have striking similarities with K3 surfaces, for example in the shape of their Hodge diamond. Actually this story goes beyond the similarity: Hassett \cite{hassett} has introduced \emph{special} cubic fourfolds as points of a codimension $1$ locus in the moduli space of cubics, each irreducible component of which consists of cubics classified by an integer $d=\mathrm{disc}\langle h^2,T\rangle$ where $h$ is the hyperplane class and $T$ is the class of a surface not homologous to $h^2$, proving that for certain values of $d$ one can associate to the cubic a K3 surface via Torelli theorem.  

On the other hand, this relation among fourfolds and surfaces has been found in the realm of derived categories, but let us review some glossary before talking about these results.

A full triangulated subcategory $\A$ of a triangulated category $\mathcal{D}$ is \emph{admissible} if the inclusion functor admits both left and right adjoints.

Two admissible triangulated subcategories $\mathcal{A}$, $\mathcal{B}$ of $\mathcal{D}$ form a \emph{semi-orthogonal decomposition} of $\mathcal{D}$, written $\mathcal{D}=\langle \mathcal{A}, \mathcal{B}\rangle$ if the following holds:
\begin{enumerate}[leftmargin=*, align=left]
\item[$\bullet$] $\Hom_\mathcal{D}(B,A)=0 \qquad \forall A\in\mathcal{A}, B\in\mathcal{B}$
\item[$\bullet$] the smallest triangulated subcategory of $\mathcal{D}$ containing both $\mathcal{A}$ and $\mathcal{B}$ coincides with $\mathcal{D}$.
\end{enumerate}
Obviously the definition can be generalized to a finite collection of triangulated subcategories $\mathcal{D}=\langle \A_1,\ldots ,\A_n\rangle$. Let us define the full triangulated subcategories left and right orthogonal respectively to a triangulated subcategory $\A$ as follows:
\begin{align*}
^\bot\A=\{E\in \mathcal{D}\,|\,\Hom_\mathcal{D}(E,A[n])=0\mbox{ for } A\in \A, n\in\Z\}\\
\A^\bot=\{E\in \mathcal{D}\,|\,\Hom_\mathcal{D}(A[n], E)=0\mbox{ for } A\in \A, n\in\Z\}.
\end{align*}
If $\A$ is admissible we then get semi-orthogonal decompositions
\[
\mathcal{D}=\langle \A, ^{\perp}\!\A\rangle \qquad\quad \mathcal{D}=\langle \A^\bot,\A\rangle.
\]

Kuznetsov \cite{kuznetsov} exhibited a semi-orthogonal decomposition of the derived category of the cubic fourfold $Y\subset \P^5_\C$
\[
\rmD^b(Y)=\langle \A_Y,\O_Y,\O_Y(H),\O_Y(2H)\rangle \qquad\A_Y:=\langle \O_Y,\O_Y(H),\O_Y(2H)\rangle^\bot
\]
where $H$ is the pullback of the hyperplane class in $\P^5_\C$, and found that in some cases one can produce a K3 surface $S$ such that $\A_Y\simeq \rmD^b(S)$, also proving that for cubic fourfolds containing a plane there is an equivalence of $\A_Y$ with the derived category of a K3 twisted by a Brauer class. Recently Addington and Thomas \cite{addingtonthomas} proved that generically the special cubic fourfolds found by Hassett and Kuznetsov coincide (and it is conjectured that they coincide globally).

It is an interesting fact that no cubic fourfold has yet been shown to be irrational, and for all the known rational examples this connection with K3 surfaces holds. It is nonetheless conjectured that the generic cubic fourfold will only be unirational (and cannot have an associated K3).

\subsection{Gerbes and twisted sheaves}

The most concise way to introduce gerbes over a scheme $X$ is saying that they are classified up to equivalence by the second cohomology group $\H^2(X,\O_X^\ast)$. This definition gives the hint that they are some kind of higher order equivalent of line bundles but it is after some work that one can characterize them as a particular kind of stacks. Stacks here will be viewed as (pseudo)functors from some algebro-geometric site $\mathcal{C}_{X}$ of $X$ to the category of groupoids, and a \emph{gerbe} will then be a stack $\mathcal{X}$ which is locally non-empty, that is for every $Y\in\mathcal{C}_{X}$ there is an open cover $U\to Y$ such that $\mathcal{X}(U)\neq \emptyset$, and such that for any two sections $\alpha$, $\beta$ in $\mathcal{X}(Y)$ there exists a cover $\{U_i\to Y\}$ such that there is at least one arrow (which is an isomorphism) $\alpha_{\vert U_i} \to \beta_{\vert U_i}$ for every $i$.

A lot of useful information about the gerbe is contained in the automorphism sheaves of $\alpha\in \mathcal{X}(Y)$ 
\[
\underline{\mathrm{Aut}}_\alpha\colon (V\to Y)\mapsto \underline{\mathrm{Aut}}_{\mathcal{X}(V)}(\alpha_{\vert V})
\]
and that is the reason to introduce $\mu$-gerbes (or gerbes banded by $\mu$) for some sheaf of groups $\mu$ as those gerbes $\mathcal{X}$ such that for every $Y$ and every $\alpha\in \mathcal{X}(Y)$ there is an isomorphism $i_\alpha\colon \mu\to \underline{\mathrm{Aut}}_\alpha$ such that for another object $\beta \in \mathcal{X}(Y)$ the diagram
\[
\begin{tikzcd}[row sep=tiny]
 & \underline{\mathrm{Aut}}_\alpha \ar[dd, "\wr"]\\
\mu \ar[ur, "i_\alpha"] \ar[dr, "i_\beta"']  & \\
 & \underline{\mathrm{Aut}}_\beta
\end{tikzcd}
\]
commutes. Morphisms of $\mu$-gerbes are morphisms of stacks $\phi\colon \mathcal{X}\to \mathcal{X'}$ such that the following diagram commutes:
\[
\begin{tikzcd}[row sep=tiny]
 & \underline{\mathrm{Aut}}_\alpha \ar[dd, "\wr"]\\
\mu \ar[ur, "i_\alpha"] \ar[dr, "i_{\phi(\alpha)}"']  & \\
 & \underline{\mathrm{Aut}}_{\phi(\alpha)}
\end{tikzcd}
\]

Given a general stack $\mathcal{X}$ over $X$ we can associate to it its coarse space $\pi_0 \mathcal{X}$, given by the sheafification of the presheaf which associates to an element $U\in \mathcal{C}_{X}$ the set of isomorphism classes (i.e. the connected components) of the groupoid $\mathcal{X}(U)$. In this way $\mathcal{X}$ is always a gerbe over $\pi_0 \mathcal{X}$.

The gerbes encountered in this work will mostly be banded by the multiplicative group sheaf $\G_m$ and therefore classified by $\H^2(X,\G_m)$.

Let $\alpha\in \H^2(X,\O_X^\ast)$ be the equivalence class of a gerbe, we can represent it as a \v{C}ech $2$-cocycle $\{\alpha_{ijk}\in \Gamma(U_{ijk},\O_X^\ast)\}_{i,j,k\in I}$ over some open cover $\{U_i\to X\}_{i\in I}$ of $X$ (where we wrote $U_{ijk}$ for $U_i \times_X U_j \times_X U_k$) and use these functions to twist the glueing of sheaves over the open subspaces:

\begin{dfn}
An \emph{$\alpha$-twisted sheaf} on a scheme $X$ is the datum of a collection of $\O_X$-modules $\{F_i\}_{i\in I}$ over the open subschemes of some cover $\{U_i\}_{i\in I}$, and of glueing functions $\{\phi_{ij}\}_{i,j\in I}$, that is $\phi_{ij}\colon {F_j}_{\vert U_i \times_X U_j}\stackrel{\sim}{\longrightarrow} {F_i}_{\vert U_i \times_X U_j}$, such that for every $i,j\in I$ it holds $\phi_{ii}=\id$, $\phi_{ij}={\phi_{ji}}^{-1}$ and
\[
\phi_{ij}\phi_{jk}\phi_{ki}=\alpha_{ijk}\id_{{F_i}_{\vert U_{ijk}}}\qquad \forall i,j,k\in I.
\]
\end{dfn}
Using a different cocycle $\{\alpha'_{ijk}\}_{ijk}$ representing $\alpha$ over the same open cover gives a bijection between twisted sheaves built using the representative $\alpha_{ijk}$ and those defined using $\alpha'_{ijk}$ and this bijection depends on the choice of a $1$-cocycle $\{\lambda_{ij}\}_{ij}\in \check{C}^1((U_i),\O_X^\ast)$ whose boundary is the difference between the two $2$-cocycles. Passing to a refinement of the cover results in the same way in a bijection of twisted sheaves. This means that we can define up to non-canonical equivalence the category of $\alpha$-twisted twisted of $\O_X$-modules and (always up to non-canonical equivalence) its notable subcategories: the abelian category of quasi-coherent $\alpha$-twisted sheaves and of coherent $\alpha$-twisted sheaves. Familiar homological constructions can then be carried over in the twisted version with substantially no harm, therefore we can for instance talk about derived categories of twisted coherent sheaves on a variety. Inside this twisted derived categories we can define $\RHom$ and $\Lotimes$, precisely like in the non-twisted case with just the caveat that, for instance, the derived tensor product of complexes of $\alpha$-twisted  and $\alpha'$-twisted sheaves will be a complex of $\alpha\alpha'$-twisted sheaves, therefore these derived functors are not endofunctors anymore. Nonetheless it still holds the adjunction relation between them. All this is carried out in detail in Chapter I.2 of \cite{caldararu}.

\begin{rmk} A more abstract but equivalent approach is: given a $\G_m$-gerbe $\mathcal{X}$ on $X$, one can define quasi-coherent sheaves on $\mathcal{X}$ and every $F\in\QCoh(\mathcal{X})$ has a natural $\G_m$-action on it, so that it can be decomposed into its $i^{th}$-eigensheaves: 
\[
F=\bigoplus_i F_i
\]
we then say that $F$ is a twisted sheaf on $X$ if $F=F_1$. 
\end{rmk}

We will need the following

\begin{lemma}\label{lemdual}
$\RHom(-,\O_X)$ is an involution on the derived category of twisted perfect complexes over a smooth projective scheme $X$.
\end{lemma}
\begin{proof}
Let us have $E$ in some twisted derived category $\rmD^b(X,\alpha)$. The smoothness of $X$ implies that $E$ is perfect. We get the natural map by applying twice the adjunction:
\begin{align*}
\RHom(E,\O_X) &\stackrel{id}{\longrightarrow} \RHom(E,\O_X)\qquad\\
E\Lotimes \RHom(E,\O_X) &\longrightarrow \O_X\qquad\\
E &\longrightarrow \RHom(\RHom(E,\O_X),\O_X)\qquad
\end{align*}
and in order to prove that the latter is an isomorphism in $D^b(X,\alpha)$ it is enough to apply a standard dévissage argument by working locally, thus reducing oneself to the case of $E=\O_X$ where the duality holds (since it is non-twisted - see for instance \cite{sga6} I Proposition 7.2). The glueing data of $E$ and of the double dual correspond to each other because they do locally and in a natural way, therefore we have isomorphism of the twisted complexes as well.
\end{proof}

\begin{rmk}
The twisting cocycle $\alpha$ can be realized in several ways other than the class of a gerbe, but some are not feasible in general. Azumaya algebras are sheaves of algebras which locally are the endomorphism sheaf of a locally free sheaf and up to a certain equivalence relation constitute $\mathrm{Br}(X)$, the Brauer group of $X$, which in general injects strictly into the cohomological Brauer group $\mathrm{Br'}(X):=\H^2_{\text{ét}}(X,\G_m)_{tors}$.
In our cases however (e.g. affine or quasi-projective schemes) we rely on big results stating that we have in fact equality of these two groups (see \cite{deJ}).
\end{rmk}

\subsection{Derived algebraic geometry}
\subsubsection{Derived stacks}

It would be impossible to give here a complete a summary of what a derived stack is, because of the big amount of technical tools involved in its mere definition: we will therefore point the reader towards the two principal comprehensive references \cite{luriePhD} and \cite{HAG2}, as well as a more readable one in \cite{toendag}.

A derived stack is some kind of generalized space, where it is now fundamental to think of a space as its functor of points, that is the representable functor from some site to the category of sets given by taking Hom sets to the space. 
 In our setting we pass from "classical" algebraic geometry, whose domain is the site of commutative algebras, to derived algebraic geometry, where we consider, over some fixed base ring $k$, simplicial commutative $k$-algebras $\sCAlg_k$ and commutative differential graded algebras in non-positive degrees $\cdgAlg_k^{\leq 0}$ as derived analogues of commutative algebras (and where we can make sense of higher topologies as well). As we will be almost always working over $\C$ the homotopy theories of these two categories are equivalent. We will also have to allow spaces of morphisms between objects and not just sets (this is the higher stacks part of the picture). 
To sum up (and introduce some notation), a derived stack can be thought of as a functor (in the higher categorical sense) $F\in \dSt_k$
\[
F\colon \sCAlg_k\to \sSet
\]
which satisfies some higher analogue of the sheaf axioms in the given (higher) topology. Therefore the simplest examples of objects we are talking about are the derived affine schemes, that is the representable functors $F=\Map_{\sCAlg_k}(A,-)$ for some simplicial commutative algebra $A$.

Derived stacks one wants to do geometry with are those which have some good representability property with respect to an atlas (exactly like schemes/algebraic spaces or Deligne-Mumford/Artin stacks in the non-derived setting are the sheaves or stacks which can be reduced to representable objects when looking them on a cover) and this idea gives rise to the notion of \emph{geometric} derived stack.

In this paper we will mostly follow the approach and formalism contained in \cite{HAG2} Chapter 2.2.

\subsubsection{Moduli of objects in a dg category}

All the results and definitions in this section can be found in the paper \cite{toenvaquie}, from where some notation is borrowed as well. For a readable introduction to dg categories and their homotopy theory the reader can consult \cite{toendg}.

All along $k$ will be a commutative unitary ring of characteristic $0$ and from a certain point on we will assume $k=\mathbb{C}$. 

\begin{dfn}
A \emph{dg category} is a $k$-linear category $T$ enriched over the category $\Ch{k}$ of (co)chain complexes, that is for every pair of objects $x,y\in T$ the space $\Hom_T(x,y)$ is a cochain complex and the composition and associativity maps are maps of chain complexes. Given every dg category $T$ there is an underlying category, denoted $[T]$ or $\H^0(T)$, obtained taking the same objects of $T$ and defining $\Hom_{[T]}(x,y)$ to be the zeroth cohomology group of $\Hom_T(x,y)$.
\end{dfn}
We can define the category $\dgCat_k$ whose objects are (small) dg categories and whose morphisms are functors between them which induce maps of cochain complexes between the Hom spaces (these are called dg functors). 

The primitive examples of $k$-linear dg categories are dg algebras over $k$, seen as a dg category with one object and the algebra as endomorphism chain complex, and the category $\dgCh{k}$ whose objects are chain complexes over $k$ and where we define $\Hom_{\dgCh{k}}(X,Y)$ as $\bigoplus_{p\in\Z}\Hom_{\catname{grMod}_k}(X,Y[p])$, with the obvious differential making this set into a chain complex.

Dg functors $T^{op}\to \dgCh{k}$ are called $T^{op}$-dg-modules. These form a dg category $T^{op}\Mod$ which will be endowed with the model structure whose fibrations and weak equivalences are those morphisms of dg-modules which are objectwise respectively fibrations or weak equivalences in the projective model structure of $\Ch{k}$. We will denote $\wh{T}$ the full dg subcategory of $T^{op}$-dg-modules formed by cofibrant $T^{op}$ dg-modules. 

\begin{dfn}
An object $x$ in a model category $\bfcalC$ is \emph{homotopically finitely presented} if $\Map_{\bfcalC}(x,-)$ commutes with filtered homotopy colimits, that is: for every filtered system $\{{y_i}\}\in \bfcalC$ the natural map
\[
\colim_i \Map_{\bfcalC}(x,y_i)\longrightarrow \Map_{\bfcalC}(x,\hocolim_i y_i)
\]
is an isomorphism. In case $\bfcalC=T^{op}\Mod$ for a dg category $T$ we will speak of \emph{perfect} or \emph{compact} modules as a shorthand for homotopically finitely presented modules.
\end{dfn}

\begin{dfn}
Let $T$ be a $k$-linear dg category.
\begin{enumerate}[leftmargin=*, align=left]
\item[$\bullet$] $T$ is \emph{locally perfect} if $\Hom_T(x,y)$ is a perfect complex of $k$-modules for every objects $x,y$ in $T$.
\item[$\bullet$] $T$ has a \emph{compact generator} if $[\wh{T}]$ has a compact generator as a triangulated category (i.e. some $g\in T$ such that $\Hom_{[T]}(g[m],x)=0$ for every $m\in\Z$ implies $x=0$ which is moreover compact in the sense that for every family ${\{x_i\}}_{i\in I}$ of objects in $T$ the natural map
\[
\coprod_{i\in I} \Hom_{[T]}(g,x_i)\longrightarrow \Hom_{[T]}(g,\coprod_{i\in I}x_i)
\]
is an isomorphism).
\item[$\bullet$] $T$ is \emph{proper} if it is locally perfect and has a compact generator.
\item[$\bullet$] $T$ is \emph{smooth} if it is a perfect object in $\wh{T\Lotimes T^{op}}$ when considered as the $T^{op}\Lotimes T$-dg-module
\begin{align*}
T\colon T^{op}\Lotimes T &\longrightarrow \dgCh{k}\\
(x,y) &\longmapsto \Hom_T(x,y)
\end{align*}
\item[$\bullet$] $T$ is \emph{triangulated}\footnote{This property is sometimes called being \emph{pretriangulated} in literature, here we mostly follow the notational conventions of To\"{e}n and Vaquié} if the Yoneda embedding $h\colon T\to \whperf{T}$ is a quasi-equivalence, which by definition means that the dg functor is a quasi-isomorphism on the morphism chain complexes and induces an equivalence on the homotopy categories.
\item[$\bullet$] $T$ is \emph{saturated} if it is proper, smooth and triangulated.
\item[$\bullet$] $T$ is \emph{of finite type} if there exists a dg algebra $B$, homotopically finitely presented in the model category \catname{dgAlg} such that $\wh{T}$ is quasi-equivalent to $\wh{B^{op}}$
\end{enumerate}
\end{dfn}

Saturated dg categories are equipped with an intrinsic Serre functor: it is the dg endofunctor $S_T\colon T\to T$ associated to the cofibrant perfect $T\otimes T^{op}$-dg-module sending the element $(x,y)\in T\otimes T^{op}$ to $\Hom_T(x,y)^\vee$, the dual chain complex over $k$ (see \cite{toenvaquiepoints}).

\begin{dfn}
Given a dg category $T$ define the functor
\begin{align*}
\scrM_T\colon \catname{sCAlg}_k &\longrightarrow \sSet\\
A &\longmapsto \Map_{\dgCat_k}(T^{op},\whperf{A})
\end{align*}
where the mapping space in $\dgCat_k$ is computed using the model structure due to Tabuada \cite{tabuada}.
\end{dfn}

\begin{thm}[\cite{toenvaquie}]
The functor $\scrM_T$ is a derived stack.

If $T$ is a smooth and proper dg category then $\scrM_T$ is a locally geometric derived stack locally of finite presentation, in the sense of the following definition.
\end{thm}

\begin{dfn}
$F\in \dSt_k$ is \emph{locally geometric} if it can be written as a filtered colimit $F\simeq \hocolim_i F_i$ such that every $F_i \in \dSt_k$ is $n_i$-geometric for some $n_i$ and each morphism $F_i \to F$ is a monomorphism\footnote{By definition this means that $F_i \to F_i \times^{h}_{F} F_i$ is a natural isomorphism in $\dSt_k$.}. If every $F_i$ can be chosen to be locally of finite presentation we say the same for $F$.
\end{dfn}

We pass freely from $\catname{sCAlg}_k$ to $\cdgAlg^{\leq 0}_k$ by means of the normalization (Dold-Kan) functor $N$ since we are in characteristic $0$, and we will usually prefer the dg formalism.

For any $F\in \dSt_k$ we can choose a presentation $F=\hocolim_i h_{A_i}$ where $A_i\in \cdgAlg^{\leq 0}_k$ and $h_{A_i}\colon A\mapsto \Hom(A_i,A)$, which by virtue of standard resolutions and straightenings can be chosen to be functorial. We then define
\[
\rmL_{perf}(F)= {\left(\holim \wh{(A_i)}_{perf}\right)}^{op}
\]
\begin{prop}[\cite{toenvaquie}]\label{adj}
There is an adjunction 
\[
\rmL_{perf}\colon \dSt_k \rightleftarrows (\dgCat_k)^{op}\colon\scrM_{-}
\]
\end{prop}

This adjunction follows formally from the definition of $\scrM_T$, Yoneda Lemma and commutation of colimits with mapping spaces, once we have given the correct definition of $\rmL_{perf}$. For caution about contravariancy, let us see how morphisms are mapped via these functors.
Given a diagram of simplicial commutative algebras $I\to \catname{sCAlg}_k$, $i\mapsto A_i$, we get a diagram of dg algebras $i \mapsto NA_i$. For a morphism $i\stackrel{\phi}{\to}j$ in $I$ we have the extension of scalars functor $NA_i\to NA_j$ which gives a dg functor
\[
NA_j\otimes_{NA_i}-\colon NA^{op}_i\Mod \longrightarrow NA^{op}_j\Mod
\]
restricting to a dg functor $\wh{A_i}\to\wh{A_j}$.

We recall the following definition:
\begin{dfn}
Let $A\in \cdgAlg^{\leq 0}_k$. An $A$-dg-module $P$ is of \emph{Tor amplitude contained in $[a,b]$} if for every $\H^0(A)$-module $M$ it holds
\[
\H^i(P\Lotimes_{A}M)=0\qquad\mbox{ if }i\notin [a,b].
\]
\end{dfn}

See section $2.4$ in \cite{toenvaquie} for more properties of this notion.

\subsubsection{Systems of points and moduli of point-like objects}\label{secpts}

This subsection is a quick recap of what we are going to use from the article \cite{toenvaquiepoints}, where they introduce t-structures on dg categories, discuss the notion of system of points and prove a result of reconstruction of a dg category with these structures as the dg category of perfect (twisted) complexes over the moduli space of point-like objects in the category.

From this section on, we will work over $\C$ and we will follow the notational convention of \cite{toenvaquiepoints} where they denote by $\scrM$ the truncated (underived) stack and write $\R\scrM$ for its derived version.

\begin{dfn} Given a saturated dg category $T$, a \emph{t-structure} on $T$ is the datum of a t-structure in the standard sense on the triangulated category $[\wh{T}]$, with the additional hypothesis that the subcategory $[\wh{T}]^{\leq 0}$ is generated under arbitrary sums, cones and direct factors by a small set of objects.
\end{dfn}

Whenever we will talk about t-structures on dg categories, these categories will always be supposed to be saturated, and in particular triangulated.

\begin{rmk} If $[T]$ already has a t-structure given by a full triangulated subcategory $[T]^{\leq 0}$, this corresponds to a full dg subcategory $T^{\leq 0}$. The homotopy colimit completion $\wh{T}$ of $T$ can then be endowed with an induced t-structure defined by letting $\wh{T}^{\leq 0}$ be the smallest full dg subcategory of $\wh{T}$ closed by colimits and containing the image of $T^{\leq 0}$ under the Yoneda embedding. With such a construction and defining the corresponding subcategory $\wh{T}^{\geq 1}$ by orthogonality, the Yoneda functor becomes a left and right t-exact embedding and $\wh{T^{\leq 0}}=\wh{T}^{\leq 0}$. The heart of the induced t-structure on $[\wh{T}]$ is ${\wh{T}}^\heartsuit=\wh{T^{\heartsuit}}$. This is now written in Lemma A.1 of \cite{benjaminmaurogabriele}.
\end{rmk}

\begin{dfn}\label{tstruct}
A t-structure on a triangulated dg category $T$ is 
\begin{enumerate}[leftmargin=*, align=left]
\item[$\bullet$] \emph{bounded} if every compact object of $[\wh{T}]$ is bounded (i.e. every compact -in the triangulated sense- $E\in [\wh{T}]$ is contained in $\wh{T}^{[a,b]}$ for some $a\leq b$);
\item[$\bullet$] \emph{precompact} if for every $E\in [T]$ the objects $\tau_{\leq 0} (E)\in [\wh{T}]^{\leq 0}$ and $\tau_{\geq 0}(E)\in [\wh{T}]^{\geq 0}$ are compact in $[\wh{T}]$;
\item[$\bullet$] \emph{compact} if it is bounded, precompact and if the full sub dg category $\wh{T}^{\geq 0}$ is stable under filtrant colimits.
\end{enumerate}
\end{dfn}

A compact t-structure over a dg category $T$ (therefore really over the colimit completion $\wh{T}$ of $T$) restricts to a t-structure in the standard sense over $[T]$.

Given a commutative $\C$-algebra we can get a saturated base-changed dg category $T_A:=(\wh{T\Lotimes_\C A})_{perf} $ and define $\wh{T_A}^{\geq 0}$ by base-change

\[
\begin{tikzcd}
\wh{T_A}^{\geq 0} \ar[d] \ar[r] & \wh{T_A} \ar[d] \\
\wh{T}^{\geq 0} \ar[r] & \wh{T}
\end{tikzcd}
\]
where the right vertical arrow is given by forgetting the $A$-module structure via $\wh{T_A}\simeq\RHom(A,\wh{T})\to \RHom(\C,\wh{T})\simeq \wh{T}$. Let also $\wh{T_A}^{\leq0}$ be the smallest sub-dg category of $\wh{T_A}$ stable under sums, cones and retracts and containing the image of $\wh{T}^{\leq 0}$ under the scalar extension functor $-\otimes_\C A\colon \wh{T}\to \wh{T_A}$. This provides the induced t-structure on $T_A$. 

We also have, for every pair of integers $a\leq b$, the substack $\scrM_T^{[a,b]}\subset \scrM_T$ defined as the one taking a commutative $\C$-algebra $A$ to the full subsimplicial set of $\scrM_T(A)$ formed by $T^{op}\otimes A$-dg-modules $E$ such that, for every $A\to A'$, 
\[
E\otimes_A A'\in [\wh{T}_{A'}]^{[a,b]}
\]

\begin{dfn}A t-structure on a dg category $T$ is
\begin{enumerate}[leftmargin=*, align=left]
\item[$\bullet$] \emph{open} if for every pair of integers $a\leq b$ the substack $\scrM_T^{[a,b]}\subset \scrM_T$ is representable by an open immersion;

\item[$\bullet$] \emph{perfect} if for every regular algebra $A$ the induced t-structure on $T_A$ is compact and moreover the heart ${\wh{T_A}}^\heartsuit$ is a noetherian abelian category.\footnote{this by definition means that every subobject of a $\omega$-small object is $\omega$-small, that is if an object $X\in T$ is such that $\Hom(X,-)$ commutes with colimits indexed by $\omega$ then for every subobject $Y\hookrightarrow X$ the functor $\Hom(Y,-)$ has the same property.}
\end{enumerate}
\end{dfn}

\begin{dfn}
An object $x\in T$ is a \emph{point object of dimension $d$} if
\begin{enumerate}[leftmargin=*, align=left]
\item[$\bullet$] $S_T(x)\simeq x[d]$
\item[$\bullet$] $\H^1\Hom_T(x,x)$ is a $\C$-vector space of dimension $d$ 
\item[$\bullet$] $\Hom_T(x,x)$ is quasi-isomorphic to $\mathrm{Sym}_k(\H^1(\Hom_T(x,x))[-1])$ as a dg-algebra.
\end{enumerate}

A \emph{system of points of dimension $d$} in $T$ is the data of a set $\calP$ of isomorphism classes of point objects of $T$ such that $\forall x\neq y\in \calP$ $\Hom_T(x,y)=0$ and 
\[
E\simeq 0\in T \Leftrightarrow \Hom_T(E,x)\simeq 0\quad \mbox{for every }x\in \calP
\]
\end{dfn}

\begin{dfn}
Let $T$ be a saturated dg category endowed with a perfect t-structure and a system of points $\calP$ of dimension $d$. The system of points $\calP$ \emph{cogenerates the t-structure} if
\[
E\in T^{\leq 0} \Leftrightarrow \H^i(\Hom_T(E,x))\simeq 0\quad \mbox{for all }x\in\calP \mbox{ and }i<0
\]

The system of points $\calP$ \emph{strongly cogenerates the t-structure} if it cogenerates the t-structure and moreover for every object $E\in T^\heartsuit$
\[
E\simeq 0 \Leftrightarrow [E,x]\simeq 0\quad \forall x\in\calP
\]
\end{dfn}

\begin{dfn}
A system of points $\calP$ is \emph{bounded} if there exists an open substack of finite type $U\subset \scrM_T$ such that
\[
\calP\subset \pi_0(U(\C))\subset \pi_0(\scrM_T(\C))
\]
\end{dfn}

When a saturated dg category is provided with a t-structure and a system of points which satisfy all the properties above, it is possible to consider the moduli stack $\scrM_\calP$ classifying objects of $\calP$, defined as the one whose space of $A$-points for $A \in \CAlg_k$ of finite type is 
\[
\scrM_\calP(A)=\{E \in T^\heartsuit\,|\,\forall a\colon A\to \C,\,E\otimes_A \C\in \calP\}.
\]
For a general algebra $A$ we define $\scrM_\calP(A)=\colim_{A_\alpha}\scrM_\calP(A_\alpha)$ where the colimit ranges over subalgebras $A_\alpha\subset A$ of finite type. We will also consider its coarse moduli space $M_\calP$. The main result in the paper \cite{toenvaquiepoints} is the following:
\begin{thm}[\cite{toenvaquiepoints} 7.2] Let $T$ be a saturated dg category equipped with a perfect and open t-structure and a bounded system of points $\calP$ such that $\calP$ strongly co-generates the t-structure. If $M_\calP$ is proper over $k$ then there exists a natural dg functor to the dg category of perfect complexes twisted by the class of the gerbe $\scrM_\calP$
\[
T^{op}\longrightarrow \Perf(M_\calP,[\scrM_\calP])
\] 
which is a quasi-equivalence.
\end{thm}

\section{Moduli of point-like objects in the dg category of a variety}

In this section we show that if $T$ is the dg category of twisted perfect complexes on a smooth projective variety $X$ over $\C$, the skyscraper sheaves form a system of points in $\Perf(X,\alpha)$. Moreover the variety can be reconstructed from such a system of points. This statement is contained in \cite{toenvaquiepoints}.

We can rephrase the definition of $\R\scrM_T$ as in section 3.5 of \cite{toenvaquie} to see that for any $A\in \cdgAlg^{\leq 0}_\C$
\begin{align*}
\R\scrM_T(A)&\simeq \bfMap_{\dSt}(X,\RPerf)(A)\simeq\\
&\simeq \left\{\begin{aligned}
&\mbox{nerve of the non-full subcategory of $\Perf(X\times^h \RSpec A)$} \\
& \,\, \mbox{ whose morphisms are weak equivalences}
 \end{aligned}\right\}.
\end{align*}

\begin{rmk}[Skyscraper sheaves are perfect of Tor amplitude $0$]\label{skyscr} 
$X$ is smooth projective, therefore $\rmL_{qcoh}(X)$ is generated by $E=\bigoplus_{0\leq i\leq d} \O(-i)$ and we can consider the map of dg categories $\C\to \Perf(X)$ sending the only object in the dg category associated to $\C$ to $E$. Like in Section 3.5 of \cite{toenvaquie}, this induces by pullback the map
\begin{align*}
\scrM_{\Perf(X)}\simeq \bfMap(X,\RPerf)&\to \scrM_{\C}=\RPerf\\
E &\mapsto \bigoplus \R\Gamma(X,E(-i))
\end{align*}
and as $X$ is smooth $\rmD^b(\Coh(X))=\rmD_{perf}(X)$ therefore we can take a perfect complex $F$ quasi-isomorphic to a skyscraper sheaf $A$.
$\O(-i)$ is locally free hence $F(-i)$ is calculated degreewise and is quasi-isomorphic to $A(-i)$. It follows $\R\Gamma(X,F(-i))=\R\Gamma(X,A(-i))$. The map above sends in this way $A$ to $\bigoplus \R\Gamma(X,A(-i))$. If $A=x_\ast(\C)$ is the skyscraper supported on the point $x$ of $X$, the projection formula gives
\[
x_\ast(\C\otimes_\C x^\ast\O(-i))\simeq x_\ast(\C)\otimes_{\O_X}\O(-i)
\]
and $x^\ast\O(-i)=x^{-1}(\O(-i))\otimes_{\O_{X,x}}\C$. $\O(-i)$ is locally free, therefore $x^{-1}(\O(-i))\simeq \O_{X,x}$ as $\O_{X,x}$-module and therefore $x^\ast(\O(-i))=\C$. It follows
\[
\R\Gamma(X,A(-i))=\R\Gamma(X,A)=\C[0].
\]
This means that $\bfMap(X,\RPerf)\to \RPerf$ sends skyscrapers of $X$ to perfect complexes in $\dgVect_\C$ concentrated in degree zero. As $\C$ is a field, those have Tor amplitude contained in $[0,0]$.

In the twisted case the result still holds due to a corollary to the main result in \cite{toenazumaya} which generalized to twisted derived categories the existence of a generator.\end{rmk}

If $X$ is a smooth projective variety and $T=\Perf(X,\alpha)$ is its dg category of perfect $\alpha$-twisted complexes, the standard t-structure on $[T]=\rmD^b(X,\alpha)$ specified by the full triangulated subcategory of $[T]$ of complexes quasi-isomorphic to ones concentrated in non-positive degrees gives a t-structure on $T$ in the sense of definition \ref{tstruct}.

\begin{lemma}
The standard t-structure on $\Perf(X,\alpha)$ is open and perfect.
Let $\calP$ be the set of skyscraper sheaves supported on points of $X$, $\mathrm{dim}(X)=d$, then $\calP$ forms a system of points in $T=\Perf(X,\alpha)$ which is bounded and strongly co-generates the standard t-structure on $\Perf(X,\alpha)$.
\end{lemma}
\begin{proof}
Let $A$ be a regular $\C$-algebra, if $T=\Perf(S,\alpha)$ then $T_A$ is saturated and therefore by definition compact objects in $[\wh{T_A}]$ correspond to perfect complexes on $\Spec A\times S$ which are $1\boxtimes \alpha$-twisted. By definition of the t-structure these will be of bounded amplitude and this shows that the t-structure is bounded. Precompactness of the t-structure on $T_A$ follows from the properties of the truncation functor 
. Stability of $T^{\geq 0}$ under filtrant colimits is readily seen by calculating of the mapping spaces (taking colimits out), and the heart of $T_A$ consists of coherent sheaves on $\Spec A\times S$ where $A$ is noetherian (regular) and $S$ is smooth projective, therefore it is noetherian abelian. This shows that the t-structure on $T$ is perfect.

The t-structure considered on $T$ is also open since this is equivalent to the semi-continuity theorem for cohomology of coherent sheaves. 

It is clear that skyscraper sheaves belong to the heart of the t-structure on $\Perf(X,\alpha)$ (by virtue of Remark \ref{skyscr}) and satisfy the cohomological properties in order to be point objects of dimension $d$ and that $\calP=\{\O_x\,|\,x\in X(\C)\}$ defines a system of points in $\Perf(X,\alpha)$. Also taking Hom's to skyscraper has the properties required to show that $\calP$ strongly co-generates the standard t-structure on perfect complexes. This system of points is also bounded because having Tor dimension contained in a certain range is an open property and therefore we can choose $U=\scrM_T^{[0,0]}$ to satify the definition of bounded system of points.
\end{proof}

As recalled in section \ref{secpts}, the machinery of \cite{toenvaquiepoints} yields the moduli space $\scrM_\calP$ whose space of $A$-points for a regular algebra $A$ is formed by perfect complexes $E$ in $\Perf(\Spec A \times S,1\boxtimes\alpha)^\heartsuit$ such that for every $A\to \C$ it holds $E\otimes_A\C\in\calP$.

If we note $p,q$ the projections from $\Spec A\times X$ to, respectively, $\Spec A$ and $X$ and $j_a\colon \{a\}\times X\hookrightarrow \Spec A\times X$ we see that $E\otimes_A \C \simeq E\otimes_{p^\ast \O_{\Spec A}} p^\ast \O_a \simeq {j_a}^\ast E$ since pullback commutes with tensor product. 

The adjunction in Proposition \ref{adj} always provides the unit map $\eta\colon X\to\scrM_{\Perf(X,\alpha)}$. 

\begin{ex}
Let us compute the image of the points of $X$ under this map.\\
Affine case: let $X=\Spec B$ be an affine scheme, then $L_{perf}(X)=\whperf{B}=\Perf(B)$ is the dg category of perfect complexes of $B$-modules and let us denote it by $T$.
Then the adjunction
\[
\Map_{(\dgCat_k)^{op}}(T,T)=:\scrM_T(B)\stackrel{\sim}{\longrightarrow}\Map_{\dSt_k}(X,\scrM_T)
\]
maps $\id_T\colon T\to T$, interpreted as an element in $\scrM_T(B)$, to the map $\Phi$ defined by
\begin{align*}
X(A)=\Hom_{\CAlg_k}(B,A)&\stackrel{\Phi_A}{\longrightarrow} \scrM_T(A)\\
f &\longmapsto (\hat{f}\circ -\colon \whperf{B}\to \whperf{A})
\end{align*}
where we denoted by $\hat{f}$ the extensions of scalars functor from $B$-modules to $A$-modules via $f$.
A point $x\in X(k)$ corresponds to a morphism $\phi_x\colon B\to k$ and therefore to a choice of a splitting $B\simeq k\oplus m_x$ where $m_x=\ker \phi_x$. In this case the map $\phi_x$ quotients out $m_x$ and we have
\[
\Phi_k(\phi_x)=\nfrac{B}{m_x}\otimes -\colon \Perf(B)\to\Perf(k)=\wh{\C}_{perf}\subset \dgCh{\C}
\]
that is the functor corepresented by the skyscraper sheaf $\O_x$.
When considering a presentation of $X$ as a colimit of affine schemes this turns out to be the skyscraper sheaf of a global point of $X$. 
\end{ex}

In \cite{toenmorita} the map $\eta$ is specified to be 
\begin{align*}
X(A)=\Hom(\Spec A, X)&\to \lvert\Perf(X\times \Spec A,1\boxtimes \alpha)\rvert \simeq \scrM_T(A)\\
(f\colon \Spec A \to X)&\longmapsto E
\end{align*}
where $E$ is a kernel of the integral transform $f^\ast\colon \Perf(X,\alpha)\to \Perf(A)$. This implies that $\eta$ factors through the point-like objects substack $\scrM_\calP$.

\begin{lemma}\label{lemflat}
Let $E\in \rmD^b(\Spec A\times X,1\boxtimes\alpha)$ such that for every $a\in \Spec A(\C)$ the complex $\L j_a^*E$ is a twisted $j_a^\ast\alpha$-sheaf on $\{a\}\times X$ (it is concentrated in degree zero). Then $E$ is isomorphic to an $\alpha$-sheaf which is faithfully flat over $\Spec A$.
\end{lemma}
\begin{proof}
This is the twisted version of Lemma 3.31 in \cite{huybrechtsFM}, from where we adapt the proof. 

As we consider the derived category of the abelian category of twisted sheaves there is the spectral sequence $E_2^{p,q}=\calH^p(\L j_a^\ast \calH^q(E))$ converging to $\calH^{p+q}(\L j_a^\ast E)$, which vanishes for $p+q\neq 0$. 

Let $m\in \Z$ be maximal such that $\calH^m(E)\neq 0$. Then at least one stalk of that sheaf is non-zero on some closed point $a\in\Spec A$, that is $E_2^{0,m}=\calH^p(\L j_a^\ast \calH^m(E))\neq 0$, and this stays there in the limit which is non-trivial only for $(p,q)=(0,0)$, forcing $m=0$. Therefore $\calH^p(E)=0$ for all strictly positive $p$.

The same reasoning applies to $E_2^{0,-1}=\calH^{-1}(\L j_a^\ast \calH^0(E))$ which is therefore 0 and this ensures the flatness of $\calH^0(E)$.

To show that the sheaves $\calH^q(E)=0$ for $q<0$ suppose instead there are some non-trivial ones for some $q$ and let $m$ be the maximal value of $q$ where this happens, and let $a\in\Spec A$ be a closed point in the support of $\calH^m(E)$. All $E^{-p,q}_2=\calH^{-p}(\L j_a^\ast \calH^q(E))$ are trivial for $q>m$ and $p<0$ we would get in the limit $\calH^m(\L j_a^\ast E)=\calH^0(\L j_a^\ast \calH^m(E))\neq 0$, a contradiction. Examining the stalks one easily deduces the faithfulness of $-\otimes_A E$.
\end{proof}

\begin{lemma}\label{axioms}
If $E\in \scrM_\calP(A)$ then $E\otimes p^\ast-$ induces a bijection between closed subschemes of $\Spec A$ and quotients of $E$ in $\Coh(\Spec A\times X,1\boxtimes \alpha)$.
\end{lemma}
\begin{proof}
Given a quotient $\O_{\Spec A}\twoheadrightarrow \O_{\Spec A/I}$ we get the quotient $E\simeq E\otimes p^\ast \O_{\Spec A} \twoheadrightarrow E\otimes p^\ast \O_{\Spec A/I}$ by right-exactness of the tensor product. 

Viceversa given a quotient $E\twoheadrightarrow F$ define $Z=\supp{p_\ast F}$, the support of $p_\ast F$. It is a closed subscheme of $\Spec A$ and therefore of the form $\Spec A/I $ for some ideal $I$ in $A$.

Write $\mathcal{G}$ for $p_\ast(E\otimes p^\ast \O_{\Spec A/I})$ over $\Spec A$, the map $E\otimes p^\ast-$ is injective if $\O_{\Spec A/I}=\O_{\supp\mathcal{G}}$. Comparing the two exact sequences
\[
0\to \mathcal{I}\to \O_{\Spec A}\to \O_{\Spec A/I}\to 0\qquad 0\to \mathrm{Ann}\mathcal{G}\to \O_{\Spec A}\to \O_{\supp\mathcal{G}}\to 0
\]
we see that $\mathcal{G}$ is the sheaf $E/IE$ over $\Spec A$ and we are then reduced to prove that $\supp_A E/IE$ equals $\Spec A/I$, that is $\mathrm{Ann}_A E/IE=I$. 

The ideal $I$ is obviously contained in $\mathrm{Ann}_A E/IE$ and for the other inclusion let $a\in \mathrm{Ann}_A E/IE$, which is equivalent to having $(aE+IE)/IE=0$ (note that $\mathrm{Ann}_A E=\emptyset$). By flatness of $E$ over $A$, we have $(aE+IE)/IE=((aA+I)/I)\otimes_A E$. So the faithful flatness of $E$ implies $(aA+I)/I=0$ or equivalently $a\in I$.

Let us now pass to surjectivity and consider 
\begin{equation}\label{shortes}
0\to K\to E \to F\to 0
\end{equation}
a short exact sequence of coherent sheaves. If we call $Z=\supp p_\ast F$ and we tensor the sequence with $p^\ast \O_Z$ we get the long exact sequence associated to the derived functor $-\Lotimes p^\ast \O_Z$
\[
\cdots \to \mathrm{Tor}^1(F,p^\ast \O_Z) \to K\otimes p^\ast O_Z \to E\otimes p^\ast \O_Z \to F\otimes p^\ast \O_Z\to 0.
\]
The last nonzero term is isomorphic to $F$ because we have the exact sequence $0\to I_Z\to \O_{\Spec A}\to \O_Z\to 0$ to which we can apply the right-exact $F\otimes p^\ast -$, getting the exact sequence (after identifying $F\otimes \O_{\Spec A}$ with $F$)
\[
\cdots \to \mathrm{Tor}^1(F,p^\ast \O_Z) \to F\otimes p^\ast I_Z\to F\to F\otimes p^\ast \O_Z\to 0
\]
and the term $F\otimes p^\ast I_Z$ is zero (this can be checked locally and every stalk is zero either because it is out of the support of $F$, out of the support of $p^\ast I_Z$ or has the right hand part of the tensor product annihilating the left part). 
For similar reasons (comparing the supports of $K$ and $p^\ast \O_Z$ which are related by the exact sequence \ref{shortes}) it also holds $K\otimes p^\ast O_Z=0$, giving the identification between $E\otimes p^\ast \O_Z$ and $F$, thus proving surjectivity as well.
\end{proof}
Gathering all together we have the following

\begin{prop}
Let $\mathcal{X}$ be a $\G_m$-gerbe on $X$ with class $\alpha$. Then considering $T=\Perf(X,\alpha)$ and $\calP$ the system of points of skyscraper sheaves on $X$ yields an equivalence of stacks
\[
\eta\colon \mathcal{X}\to \scrM_\calP
\]
\end{prop}
\begin{proof}The map is obtained by factoring the unit map of the adjunction $\mathrm{L}_{perf} \dashv \R\scrM$ via the inclusion of $\scrM_\calP\hookrightarrow \scrM_T$ and is a monomorphism because only equal morphisms can have isomorphic graphs.
Let us consider first the case where $\mathcal{X}=\mathrm{B}\G_m\times X$ is the trivial $\G_m$-gerbe on $X$. What we have proved in Lemma \ref{axioms} means that our hypotheses imply the axioms for point objects in \cite{john} hold and in particular their Lemmas 2.7 and 2.8 yield that, writing
\[
\begin{tikzcd}\label{supp}
Z=\supp E \ar[r, hook, "\iota"] \ar[dr, "\rho"'] & \Spec A\times X \ar[d, "p"]\\
& \Spec A
\end{tikzcd}
\]
$\rho$ is an affine (universal) homeomorphism. Here $\R p_\ast E= p_\ast E$ is a perfect complex concetrated in degree zero and of constant rank 1 (as we see from watching fiberwise via $j_a\colon \{a\}\times X\hookrightarrow \Spec A\times X$).

In the sequence of maps
\[
\O_{\Spec A}\stackrel{a}{\rightarrow} \rho_\ast \O_Z \stackrel{b}{\to} \underline{End}_{\rho_\ast \O_Z}(\rho_\ast \iota^\ast E) \stackrel{c}{\to} \underline{End}_{\O_{\Spec A}}(\rho_\ast \iota^\ast E)\simeq \O_{\Spec A}
\]
we have $cba=\id$, $c$ injective (being a forgetful map) and $b$ injective (as $Z=\supp E$), therefore $a$ is an isomorphism and $\rho\colon Z\stackrel{\sim}{\to} \Spec A$. Calling $f$ its inverse we have that $E\otimes p^\ast(p_\ast E)^\vee \simeq \O_{\Gamma_f}$. This proves that the map $\eta$ is an epimorphism and therefore an equivalence of stacks.

The case when $\mathcal{X}$ is a $\G_m$-gerbe over $X$ begins by considering $E\in \Perf(\Spec A\times X, 1\boxtimes \alpha)$ and remarking that the considerations made around the map $\rho$ in the diagram \ref{supp} still hold true, that is $\rho$ is an affine universal homeomorphism and in particular this implies that every trivialization of the gerbe structure on $\supp E$ is the pullback via $\rho$ of a trivialization over $\Spec A$, therefore taking an open cover trivializing $f^\ast\alpha$ we get that on every patch the gerbe is trivial and the result holds, the global statement hence follows via descent.
\end{proof}

\begin{rmk}\label{-1}
The above equivalence is an equivalence of stacks between two $\G_m$-gerbes which is not an equivalence of $\G_m$-gerbes, because the map $\mathcal{X}\to \scrM_\calP$ does not respect the action of $\mathrm{B}\G_m$\footnote{We would like to thank Bertrand To\"en for pointing this out.}. Nevertheless, when precomposing the given action by $-1$ we have a morphism of $\G_m$-gerbes which is an equivalence. In other words, the diagram
\[
\begin{tikzcd}
\mathrm{B}\G_m\times \mathcal{X} \ar[d, "-1\times\eta"'] \ar[r] & \mathcal{X} \ar[d, "\eta"]\\
\mathrm{B}\G_m\times \scrM_\calP \ar[r] & \scrM_\calP
\end{tikzcd}
\]
commutes.
\end{rmk}

\section{Moduli of point-like objects in the dg category of K3 surfaces and of cubic fourfolds}

Recall the semi-orthogonal decomposition of the derived category of a cubic fourfold $Y$ as
\[
\rmD^b(Y)=\langle \A_Y,\O_Y,\O_Y(H),\O_Y(2H)\rangle.
\]

Let us also recall that for a K3 surface $S$ one can define the transcendental lattice $T(S)$ as the orthogonal in $\H^2(S,\Z)$ to the algebraic part of the cohomology. A Brauer class $\beta\in \mathrm{Br}(S)$ can be realised as a group morphism $T(S)\to \Q/\Z$ (via pairing with a rational $B$-field lifting $\beta$, see \cite{huybrechtsstellari}) and we define the sublattice $T(S,\beta)$ as the kernel of this morphism. In particular the order of the element $\beta$ in the Brauer group is related to the discriminants of the lattices via the formula
\[
\lvert \beta \rvert^2\cdot \lvert \mathrm{disc}(T(S,\beta))\rvert = \lvert \mathrm{disc}(T(S))\rvert.
\]
The past section can be used to prove:

\begin{thm}\label{thm}
Let $Y\subset \P^5$ be a cubic fourfold, $\A_Y$ the admissible triangulated subcategory as above and $T_Y$ a dg enhancement of the latter. The following are equivalent:
\begin{enumerate}
\item there exists a twisted K3 surface $(S,\alpha)$ such that $\rmD^b(S,\alpha)\simeq \A_Y$
\item there exists a system of points $\calP$ of dimension two in $T_Y$ strongly co-generating a t-structure on $T_Y$ and $M_\calP$ is a K3 surface.
\end{enumerate}
In particular the K3 is untwisted (i.e. $\alpha=1$) if and only if $\scrM_\calP$ is the trivial $\G_m$-gerbe on $M_\calP$.
\end{thm}

\begin{proof}
The implication $(1)\Rightarrow (2)$ is the content of the previous section. 

For the converse we use Theorem 7.2 in \cite{toenvaquiepoints} which gives the equivalence between $T_Y^{op}$ and the dg category of perfect complexes on $M_\calP$ twisted by $[\scrM_\calP]\in H^2_{\text{ét}}(S,\G_m)$. We use Lemma \ref{lemdual} and dualize in order to get
\[
T_Y \simeq \Perf(M_\calP,[\scrM_\calP]^{-1})
\]

Now it is clear that the triviality of the gerbe $\scrM_\calP$ gives the trivial twist $\alpha=1$, while for the converse it is sufficient to remark that there is no equivalence $\rmD^b(S)\simeq \rmD^b(S,\beta)$ because of the short exact sequence involving the transcendental lattices (see \cite{huybrechtsstellari})
\[
0\to T(S,\beta) \to T(S)\to \nfrac{\Z}{n\Z}\to 0
\]
where $n$ is the order of $\beta$ in the Brauer group: having nontrivial $\beta$ would give that the discriminants of $T(S)$ and $T(S,\beta)$ differ by the constant $n^2\neq 1$ and therefore the lattices cannot be isometric. Hence the twisted and untwisted derived categories cannot be equivalent.
\end{proof}

\begin{rmk}
The fact of taking the duals in the proof of the theorem reflects the content of Remark \ref{-1}. Indeed the equivalence of stacks $\mathcal{X}\simeq \scrM_\calP$ does not induce a well defined morphism of dg categories of twisted sheaves $\Perf(S,\alpha)\simeq \Perf(M_\calP, [\scrM_\calP])$ because this is not a map of $\G_m$-gerbes. The trick of reversing the action of $\G_m$ over one of the two gerbes makes then appear the inverse of the twist in one of the two sides, giving a well defined equivalence.
\end{rmk}

\subsection{An alternative approach}\label{pt0}

Recall that for a K3 surface $S$ there is a linear map from the Grothendieck group $K_0(S)$ to the even part of the cohomology with coefficients in $\Q$ sending the class of a complex of sheaves to its Chern character (which can be easily extended on complexes, see for instance \cite{huybrechtsFM}). 

We want to characterize points on $S$ in a different manner, so we try to isolate skyscraper sheaves (and in general structure sheaves of graphs of morphisms into $S$) from a homological point of view: let $x\in S(\C)$ be a closed point and let $\O_x$ be the skyscraper sheaf  supported on $x$. The same symbol will also denote the image of this sheaf in the derived category. We will also often drop the reference to the base field $\C$.


We will therefore consider for $A\in\sCAlg$ and $E\in \Perf(S\times^h \RSpec A)$ the following conditions:
\begin{align*}
\left\{
\begin{aligned}
&\ext^i(E,E)=\ext^i(\O_x,\O_x)\qquad \mbox{for every $i\in\Z$}\\
&\ch(E)=\ch(\O_x)
\end{aligned}\right.
 \tag*{[$\spadesuit$]} \\
E\mbox{ has Tor amplitude contained in }[0,0] \tag*{[0]}
\end{align*}

Consider the full substack $\R\scrM_T^{pt,0}$ of $\R\scrM_T$ defined by setting $\R\scrM_T^{pt,0}(A)$ to be the full subsimplicial set od $\R\scrM_T(A)$ consisting of all the $E\in\R\scrM_T(A)$ satisfying [$\spadesuit$] and [$0$], for $A\in\sCAlg$.
\begin{prop}
$\R\scrM_T^{pt,0}$ is an open substack of $\R\scrM_T$.
\end{prop}
\begin{proof} We will show that the truncated (underived) stack $\scrM_T^{pt,0}:=t_0(\R\scrM_T^{pt,0})$ is an open substack of $t_0(\R\scrM_T)$. The equivalence of Zariski and étale sites of $\R\scrM_T$ and $\scrM_T$ in \cite{schurgtoenvezzosi} proves that this suffices.

Chern classes are functorial, therefore for any thickening $S\times \Spec A_0 \stackrel{i}{\hookrightarrow} S\times \Spec A$ induced by a square-zero extension $A\to A_0$, and for any $\mathcal{E}\in \Perf{S\times \Spec A}$ flat over $\Spec A$ such that $E:= i^\ast \mathcal{E}$ we have $c_k(E)=c_k(i^\ast \mathcal{E})=i^\ast c_k(\mathcal{E})$ and $i^\ast\colon \H^\bullet(S\times \Spec A,\Z)\to\H^\bullet(S\times \Spec A_0,\Z)$ is an isomorphism since their underlying topological spaces are the same. The Chern character is then a function of the Chern classes hence $\ch([\mathcal{E}])=\ch([E])$ when we identify the cohomology groups of the space and its thickening.

We can rephrase the conditions on the ext's by saying that the trace from $\Ext^0(E,E)$ to $k$ is an isomorphism and that $\chi(E,E)=0$, and both these conditions are stable under small deformations in the setting of the preceding point.%
\end{proof}

The results of \cite{toenmorita} are used in \cite{toenvaquie} to calculate the homotopy groups of the simplicial set $\R\scrM_T(A)$ and this gives us 
\begin{align*}
\pi_i(\R\scrM_T^{pt,0}(A),E)\simeq \Ext^{1-i}(E,E)=0 \qquad \mbox{if i>1}\\
\pi_1(\R\scrM_T^{pt,0}(A),E)\simeq \Aut_{\Ho(T\Lotimes A-Mod)}(E,E)\simeq \C^\ast
\end{align*}
with $A\in\cdgAlg^{\leq 0}_k$, $E\in T\Lotimes A-Mod$. This implies that $\R\scrM_T^{pt,0}$ takes $1$-truncated values and can therefore be considered as a geometric $1$-stack and $\R\scrM_T^{pt,0}\to \pi_0 t_0$ $ \R\scrM_T^{pt,0}$ is a fibration on the coarse moduli space with vertex groups $\G_m$. This is obviously locally nonempty and locally connected by properties of perfect complexes. It follows that $\R\scrM_T^{pt,0}$ is a $\G_m$-gerbe over the coarse space $M_T^{pt,0}:=\pi_0 t_0$ $\R\scrM_T^{pt,0}$. 

\begin{prop}\label{mtpt0}
The map $\eta\colon S \to M_T^{pt,0}$ obtained from the unit of the adjunction \ref{adj}, factored through the truncation and coarse space of $\scrM_T^{pt,0}$, 
is an open immersion.
\end{prop}
\begin{proof}
This is an almost word by word translation in the algebraic setting of Proposition 5.6 in \cite{tvalgebrisation}.

Let $f,g\colon \Spec A \to S$ such that $\eta_A(f)= \eta_A(g)$ in $\scrM_T^{pt,0}(A)$. This means that the structure sheaf of their graphs $\Gamma(f)$ and $\Gamma(g)\subset S\times \Spec A$ are isomorphic in $\Coh(\Spec A\times S)$, at least locally, but since being equal is a local property on $\Spec A$ we can suppose that the structure sheaves of the graphs of $f$ and $g$ are isomorphic. This implies that the algebraic subspaces defined by those are equal and therefore the morphisms must be equal as well.
This means that the algebraic subspaces defined by them are equal and the map is therefore a monomorphism.

Proving that $\eta$ is formally smooth implies that it is formally étale because it is a monomorphism. Let $A\to A_0$ be a square-zero extension of commutative $\C$-algebras and let
\begin{equation}
\begin{tikzcd}\label{etasmooth}
\Spec A_0 \ar[r, "u"] \ar[d, hook, "i"'] & S \ar[d, hook, "\eta"] \\
\Spec A \ar[r, "v"'] & M_T^{pt,0}
\end{tikzcd}
\end{equation}
we want to check the existence and uniqueness of $w\colon \Spec A\to S$ making the above diagram commute. Given a morphism $u\colon \Spec A_0\to S$ we will denote by $E_0(u):= \O_{\Gamma_u}$ the structure sheaf of the graph $\Gamma_u\subset \Spec A_0\times S$ of $u$.

The data of a diagram such as (\ref{etasmooth}) is equivalent to giving a pair $(u,E)$ where $u\colon \Spec A_0\to S$ is a morphism and $E\in \Perf(\Spec A \times S)$ is such that $\L i^\ast(E)\simeq E_0(u)$ in $\Perf(\Spec A_0\times S)$.

We are therefore looking for a $w\colon \Spec A\to S$ such that $E(w):=\O_{\Gamma_u}\simeq E$. The condition $w\circ i=u$ will be automatically implied by $\eta$ being a monomorphism. 
Such an $E\in \rmD_{perf}(\Spec A \times S)$ with the property that $\L i^\ast(E)$ is coherent and flat over $\Spec A_0$ is coherent and flat over $\Spec A$ (by \cite{aneltoen}). 

Let $q\colon \Spec A \times S \to S$ be the projection, for all $a\in \Spec A$, letting $j_a\colon \{a\} \times  S\to \Spec A\times S$ the inclusion,
\[
\L j_a^\ast \R q_\ast (E)\stackrel{qis}{\simeq} \C[0]
\]
therefore $\R q_\ast(E)=: L$ is a line bundle over $\Spec A$. By restricting to sufficiently small open subsets we can suppose $L$ trivial. Therefore there is $e\colon \O_{\Spec A}\simeq \R q_\ast(E)$, which by adjunction, corresponds to an $f\colon q^\ast\O_{\Spec A}=\O_{\Spec A\times S}\to E$. Since $\L i^\ast(E)\simeq E_0(u)$, $f$ is an epimorphism at each point in $\Spec A\times S$ and therefore $E\simeq \O_Z$ for some $Z\subset \Spec A\times S$, algebraic and flat over $\Spec A$. $Z\times_{\Spec A}\Spec A_0\simeq \Gamma_u\to \Spec A_0$ is an isomorphism, therefore $Z\to \Spec A$ is an isomorphism. Hence there exists $w\colon \Spec A\to S$ whose graph is $Z$ and we have $E(w)\simeq \O_Z\simeq E$.
\end{proof}

\begin{rmk} We observe that on closed points the map $\eta_\C\colon S(\C)\to M_T^{pt,0}(\C)$ is surjective because taking $E\in \Coh(S)$ with the prescribed Chern character gives (for $L=i^\ast(\O(1))$, $i\colon S\hookrightarrow \P^m$)
\[
P_E(n)=\chi (S,E(n))=\int_S \ch(E\otimes L^n). \td(S)\equiv 1
\]
therefore $E$ is supported in dimension zero i.e. on a finite number of points, and with 1-dimensional global sections. That is exactly a skyscraper supported over one closed point.

In particular we also have $M_\calP(\C)\simeq M_T^{pt,0}(\C)$, so they have the same closed points but their infinitesimal information might be different.
\end{rmk}

\begin{rmk} In this work we mostly stuck to the case where $\A_Y$ is already known to be geometric, but the ultimate goal would be having a general treatment of those K3 categories independently of their geometric incarnation. This however clashes with some obstacles, first of all that the general $\A_Y$ does not contain enough point objects. Another big problem is that it is not known to date (to the best of the author's knowledge) if $\A_Y$ possesses a unique dg enhancement, as is the case for most derived categories of schemes, and this would make more difficult to develop the theory. Last but not least, another fundamental ingredient would be missing in the non-geometric case but it is only recently that a non-trivial t-structure has been claimed to have been produced on $\A_Y$ by Bayer, Lahoz, Macrì and Stellari.
\end{rmk}

\vspace{1.5cm}

\begin{thebibliography}{30}

\bibitem[AneTo{\"e}09]{aneltoen}
Mathieu Anel and Bertrand To{\"e}n.
\newblock D\'enombrabilit\'e des classes d'\'equi\-va\-lences d\'eriv\'ees de
  vari\'et\'es alg\'ebriques.
\newblock {\em J. Algebraic Geom.}, 18(2):257--277, 2009.

\bibitem[AddTho14]{addingtonthomas}
Nicolas Addington and Richard Thomas.
\newblock Hodge theory and derived categories of cubic fourfolds.
\newblock {\em Duke Math. J.}, 163(10):1885--1927, 2014.

\bibitem[C{\u{a}}l00]{caldararu}
Andrei C{\u{a}}ld{\u{a}}raru.
\newblock Derived categories of twisted sheaves on {C}alabi-{Y}au manifolds.
\newblock PhD. Thesis, Cornell University, 2000.

\bibitem[CalGro15]{john}
John Calabrese and Michael Groechenig.
\newblock Moduli problems in abelian categories and the reconstruction theorem.
\newblock {\em Algebr. Geom.}, 2(1):1--18, 2015.

\bibitem[deJ03]{deJ}
Aise~J. de~Jong.
\newblock A result of gabber.
\newblock 2003. \url{http://www.math.columbia.edu/~dejong/papers/2-gabber.pdf}.

\bibitem[HAG2]{HAG2}
Bertrand To{\"e}n and Gabriele Vezzosi.
\newblock Homotopical algebraic geometry {II}. {G}eometric stacks and
  applications.
\newblock {\em Mem. Amer. Math. Soc.}, 193(902), 2008.

\bibitem[Has00]{hassett}
Brendan Hassett.
\newblock Special cubic fourfolds.
\newblock {\em Compositio Math.}, 120(1):1--23, 2000.

\bibitem[HPV16]{benjaminmaurogabriele}
Benjamin Hennion, Mauro Porta, and Gabriele Vezzosi.
\newblock Formal gluing for non-linear flags.
\newblock \href{http://arxiv.org/abs/1607.04503}{arXiv:1607.04503}.

\bibitem[HuySte05]{huybrechtsstellari}
Daniel Huybrechts and Paolo Stellari.
\newblock Equivalences of twisted k3 surfaces.
\newblock {\em Math. Ann.}, 332(4):901--936, 2005.

\bibitem[Huy06]{huybrechtsFM}
Daniel Huybrechts.
\newblock {\em Fourier-{M}ukai transforms in algebraic geometry}.
\newblock Oxford Mathematical Monographs. Oxford University Press, 2006.

\bibitem[Kuz10]{kuznetsov}
Alexander Kuznetsov.
\newblock Derived categories of cubic fourfolds.
\newblock In {\em Cohomological and geometric approaches to rationality
  problems}, volume 282 of {\em Progr. Math.}, pages 219--243. Birkh\"auser,
  Boston MA, 2010.

\bibitem[Lur04]{luriePhD}
Jacob Lurie.
\newblock Derived algebraic geometry.
\newblock PhD Thesis, MIT, 2004.

\bibitem[SGA6]{sga6}
{\em Th\'eorie des intersections et th\'eor\`eme de {R}iemann-{R}och}.
\newblock Lecture Notes in Mathematics, Vol. 225. Springer-Verlag, Berlin-New
  York, 1971.
\newblock S{\'e}minaire de G{\'e}om{\'e}trie Al\-g{\'e}\-brique du Bois-Marie
  1966--1967 (SGA 6), Dirig{\'e} par P. Berthelot, A. Grothendieck et L.
  Illusie. Avec la collaboration de D. Ferrand, J. P. Jouanolou, O. Jussila, S.
  Kleiman, M. Raynaud et J. P. Serre.

\bibitem[STV15]{schurgtoenvezzosi}
Timo Sch{\"u}rg, Bertrand To{\"e}n, and Gabriele Vezzosi.
\newblock Derived algebraic geometry, determinants of perfect complexes, and
  applications to obstruction theories for maps and complexes.
\newblock {\em J. Reine Angew. Math.}, 702:1--40, 2015.

\bibitem[Tab05]{tabuada}
Gonçalo Tabuada.
\newblock Une structure de cat\'egorie de mod\`eles de {Q}uillen sur la
  cat\'egorie des dg-cat\'egories.
\newblock {\em C. R. Math. Acad. Sci. Paris}, 340(1):15--19, 2005.

\bibitem[To{\"e}07]{toenmorita}
Bertrand To{\"e}n.
\newblock The homotopy theory of {$dg$}-categories and derived {M}orita theory.
\newblock {\em Invent. Math.}, 167(3):615--667, 2007.

\bibitem[To{\"e}11]{toendg}
Bertrand To{\"e}n.
\newblock {Lectures on dg-categories}.
\newblock In {\em {Topics in algebraic and topological $K$-theory. Papers based
  on the Sedano winter school on K-theory (Swisk), Sedano, Spain, January
  22--27, 2007}}, pages 243--302. Sprin\-ger, 2011.

\bibitem[To{\"e}12]{toenazumaya}
Bertrand To{\"e}n.
\newblock Derived {A}zumaya algebras and generators for twisted derived
  categories.
\newblock {\em Invent. Math.}, 189(3):581--652, 2012.

\bibitem[To{\"e}14]{toendag}
Bertrand To{\"e}n.
\newblock Derived algebraic geometry.
\newblock Number~2 in EMS Surv. Math. Sci. 1, pages 153--245. 2014.

\bibitem[To{\"e}Vaq07]{toenvaquie}
Bertrand To{\"e}n and Michel Vaqui{\'e}.
\newblock Moduli of objects in dg-cat\-e\-go\-ries.
\newblock {\em Ann. Sci. \'Ecole Norm. Sup. (4)}, 40(3):387--444, 2007.

\bibitem[To{\"e}Vaq08]{tvalgebrisation}
Bertrand To{\"e}n and Michel Vaqui{\'e}.
\newblock Alg\'ebrisation des vari\'et\'es analytiques complexes et
  cat\'egories d\'eriv\'ees.
\newblock {\em Math. Ann.}, 342(4):789--831, 2008.

\bibitem[To{\"e}Vaq15]{toenvaquiepoints}
Bertrand To{\"e}n and Michel Vaqui{\'e}.
\newblock Systèmes de points dans les dg-catégories saturées.
\newblock \href{http://arxiv.org/abs/1504.07748v1}{arXiv:1504.07748v1}.

\end{thebibliography}

%
\vspace{1.5cm}
\flushleft
Martino Cantadore\\
\vspace{.1cm}
\texttt{cantadore@mat.uniroma1.it}\\
\vspace{.2cm}
\small{\textsl{Dipartimento di Matematica "G. Castelnuovo",\\
 Università di Roma "La Sapienza"}\\
Piazzale A. Moro 5, 00185, Roma, Italy}\\

\end{document}